\documentclass[11pt]{article}

\usepackage{epsfig, graphicx}

\usepackage{amsthm,amsmath,amsmath,amssymb,amsthm,amsfonts,amstext,amsbsy,amscd}
\RequirePackage[colorlinks=true,citecolor=blue,urlcolor=blue]{hyperref}

\RequirePackage[OT1]{fontenc}
\usepackage[numbers, square]{natbib}

\usepackage{mathrsfs, amscd, amsfonts}

\newcommand{\field}[1]{\mathbb{#1}}
\newcommand{\R}{\field{R}}

\newcommand{\PP}{\field{P}}

\def\der^#1_#2{\frac{\partial^{#1}}{\partial {#2}^{#1}}}

\theoremstyle{plain}
\newtheorem{theorem}{Theorem}

\newtheorem{lemma}{Lemma}
\newtheorem{prop}{Proposition}
\newtheorem{assumption}{Assumption}{\bf}{\rm}
\theoremstyle{remark}

{\smallskip}%

\newcommand{\E}{\field{E}}

\begin{document}

\title{Branching processes and bacterial growth} 

\author{
 Nathalie Krell\thanks{Universit\'e de Rennes 1, CNRS-UMR 6625, Campus de Beaulieu, 35042 Rennes Cedex, France. {\tt email}: nathalie.krell@univ-rennes.fr}
}

\maketitle
\begin{abstract}
We model the growth of a cell population using a piecewise deterministic Markov branching tree. In this model, each cell splits into two offspring at a division rate $B(x)$, which depends on its size $x$. The size of each cell increases exponentially over time, with a growth rate that varies for each individual. Expanding upon the model studied in \cite{hof}, we introduce a scenario with two types of bacteria: those with a young pole and those with an old pole. Additionally, we account for the possibility that a bacterium may not always divide into exactly two offspring. We will demonstrate that our branching process is well-defined and that it satisfies a many-to-one formula. Furthermore, we establish that the mean empirical measure of the model adheres to a growth-fragmentation equation when structured by size, growth rate, and type as state variables.

\end{abstract}

{\small 
}

\section{Introduction}

We will explore the modeling of bacterial growth, with a particular focus on Escherichia coli (E. coli). To do so, we will employ branching processes. A bacterium  grows exponentially at a constant rate $\tau$. Naturally, there is variability in the growth rate $\tau$, see for example \cite{Sturm}. In a previous study \cite{BB}, we examined in detail the dependencies between bacteria, specifically how this rate might be influenced by the mother, grandmother, and even the bacterium's type. For further details, one can refer to the aforementioned paper. Indeed, certain bacteria possess an old pole, while others have a new pole; more precisely, one could be interested in how long a bacterium has belonged to a particular type. However, in this paper, we will simplify the approach by focusing solely on the bacterium's current type, considering only two types: old pole and new pole.

Each bacterium divides independently of the others, conditioned on their size at birth, and according to a specific intensity $B$. In an earlier work \cite{hof2}, we compared an age-dependent model with one that depends on size. At that time, we demonstrated that an age-dependent model was far from realistic. Estimating the intensity $B$ using this model and then regenerating a colony of bacteria with this intensity resulted in an accumulation of small bacteria, which was entirely inconsistent with experimental data. Therefore, in this paper, I will hypothesize that the intensity $B$ depends on the size of the bacteria. In ongoing research with Benoîte de Saporta and Bertrand Cloez, we are conducting a comparative study with a novel model introduced by biologists—the Adder model, where the intensity depends on the bacterium's elongation, defined as the difference between its current size and its size at birth.

When a bacterium divides, its size is divided into two parts, which may not necessarily be equal. We will assume that the old cell receives a portion $\theta_0$ from the mother, while the young cell receives a portion $\theta_1 := 1-\theta_0$. Thus, $\theta_0$ and $\theta_1$ are two positive real numbers whose sum equals 1.

The case we will study here extends the model explored in \cite{hof}, adding another layer of complexity by introducing a dependency based on the bacterium's type, and allowing for unequal division.

In Section \ref{size_models}, we construct the model $\big((\xi_u, \tau_u, p_u), u \in {\mathcal U}\big)$, which represents the sizes, growth rates, and types of cells as a Markov chain along the genealogical tree. This discrete model is then embedded into a continuous-time piecewise deterministic Markov process $(X,V,P)$, capturing the sizes, growth rates, and types of cells present in the system at any given moment. We will define the marked bacteria process, which will yield a many-to-one formula—a fundamental tool for proving Theorem \ref{sol transport general}. Notably, this marked bacterium is directly linked to our branching process and does not merely exist in law, as is often the case with many-to-one formulas. In Theorem \ref{sol transport general}, we explicitly establish the connection between the mean empirical measure of $(X,V,P)$ and the growth-fragmentation type equation \ref{transport variabilite}, extending the results obtained in \cite{hof}. In Section \ref{statistical analysis}, we explore potential applications of the aforementioned probabilistic results, focusing on the nonparametric estimation of the jump rate $B$. Finally, in the last section, we provide the proof of Theorem \ref{sol transport general}.

\section{Size-structured models } \label{size_models}
\subsection{The branching process } %

\label{size-strutured models}

We will use a genealogical tree with the Ulam-Harris-Neveu labelling. In each node we will be able to store all information we need to build our process of growth of the size of our bacteria.

We can note that we are in a simpler tree case because each division of a bacterium gives birth to two bacteria.

Let
${\mathcal U} := \bigcup_{n=0}^\infty \{0,1\}^n$ 
(with $\{0,1\}^0:=\{\emptyset\}$) denotes the infinite binary genealogical tree. Each node $u \in {\mathcal U}$ is identified with a cell of the population and has a mark 
$$(\xi_u, b_u, \zeta_u, d_u, p_u,  \tau_u, \theta_u),$$
where $\xi_u$ is the size at birth, $b_u$ the birthtime,  $d_u$ the deathtime, $\zeta_u$ the lifetime of $u$,   $p_u$ the type, $\tau_u$ the growth rate,  and $\theta_u$ the proportion of the mother cell inherited at birth.  $\xi_u, b_u, d_u, \zeta_u$ are positive real numbers. $p_0$ is 0 if the bacterium is an old pole and 1 if it's a new pole. The growth rate $\tau $ and the $\theta$ only depend on the type. Therefor $\tau_u=\tau_{p_u}$ and $\theta_u=\theta_{p_u}$. To streamline the notation, we will use $\theta_0$ to represent the parameter $\theta$ associated with the old pole, and $\theta_1$ for the parameter $\theta$ associated with the new pole.. The same applies to $\alpha$.  The $\tau_u$ will live in a compact space ${\mathcal E}
\subset\mathbb{R }$.

 The evolution $\big(\xi_t^{u}, t \in [b_u,b_u+\zeta_u)\big)$ of the size of $u$ during its lifetime is governed by 
\begin{equation} \label{piecewise deterministic}
\xi_t^u=\xi_u \exp\big(\tau_u (t-b_u)\big)\;\;\text{for}\;\;t\in [b_u,b_u+\zeta_u).
\end{equation}
Each cell splits into two offsprings of the same size according to a division rate $B(x)$ for $x\in (0,\infty)$. Equivalently 
\begin{equation} \label{def div rate}
\PP\big(\zeta_{u}\in [t,t+dt]\,\big|\,\zeta_{u}\geq t,\xi_{u}=x, \tau_u=v\big) = B\big(x\exp(v t)\big)dt.
\end{equation}
At division, a cell splits into two offsprings of the same size. If $u^-$ denotes the parent of $u$, we thus have 
\begin{equation} \label{fundamental}
\xi_u = \theta_u \xi_{u^-}\exp\big(\tau_{u^-} \zeta_{u^-}\big).
\end{equation}
Finally, the growth rate $\tau_u$ of $u$ is inherited from its parent $\tau_{u^-}$ and depending on the type $p_u=i$ according to a Markov kernel
\begin{equation} \label{markov heritage}
\rho_i (v, dv')=\PP^{i}(\tau_u\in dv'\,|\,\tau_{u^-}=v),
\end{equation}
where $v >0$ , $i\in\{0,1\}$ and $\rho_i(v, dv')$ is a probability measure on $(0,\infty)$ for each $v>0$ and  $i\in\{0,1\}$.

We denote by $y = (x,j)$ an element of the state space ${\mathcal T} = [0,\infty)\times \{0,1\}$. Introduce the transition kernel 
$$
{\mathcal\mathbb Q}_{B}((x,i),dx'\times \{i^{'}\}) = \PP\big((\xi_u,t_u)\in dx'\times \{i^{'}\}\big|\,(\xi_{u^-},t_{u^-})=(x,i)\big)
$$
of the size at birth and type $(\xi_u,p_u)$ of a descendant $u \in {\mathcal U}$, given the size at death and type of its parent $(\xi_{u^-},p_{u^-})$.
From \eqref{def div rate}, we infer that $\PP(\zeta_{u}\in dt\,\big|\,\xi_{u^-}=x, p_{u}=j,\tau_{u-}=v)$ is equal to
$$
B\big(x\exp(v t)\big)\exp\Big(-\int_0^t B\big(x\exp(v s)\big)ds\Big)dt.
$$
Using formula (\ref{fundamental}), by a simple change of variables
$$\PP\big(\xi_u\in dx' \, ,\, p_{u}=i' \big|\,\xi_{u^-}=x, p_{u^-}=i,\tau_{u-}=v\big)
=\frac{B(x')}{2v x'}{\bf 1}_{\{x' \geq x\theta_{i'}\}}\exp\big(-\int_{x \theta_{i'}}^{x'} \tfrac{B(s)}{vs}ds\big)dx'.
$$
We obtain an explicit formula for
$$
{\mathcal\mathbb P}_{B}((x,i,v),dx'\times \{i^{'}\}\times dv')  ={\mathcal P}_B\big((x,i,v), (x', i', dv'))dx',$$
with
\begin{equation} 
{\mathcal\mathbb P}_B\big((x,i,v), (x', i',dv'))  = \frac{B(x'/\theta_i)}{v x'}{\bf 1}_{\{x' \geq x \theta_{i'}\}}\exp\big(-\int_{x\theta_{i'}}^{x'} \tfrac{B(s/\theta_i))}{vs}ds\big)\rho_i(v, dv')
 .  \label{density explicit}
\end{equation}

Eq. \eqref{piecewise deterministic}, \eqref{def div rate}, \eqref{fundamental} and \eqref{markov heritage} completely determine the dynamics of the model $\big((\xi_u, \tau_u, p_u), u \in {\mathcal U}\big)$, as a Markov chain on a tree, given an additional initial condition $(\xi_\emptyset, \tau_\emptyset, p_\emptyset)$ on the root. 
The chain is embedded into a piecewise deterministic continuous Markov process thanks to \eqref{piecewise deterministic} by setting
$$(\xi^u_t, \tau_t^u, p_t^u) =\big(\xi_u \exp\big(\tau_u( t-b_u)\big), \tau_u, p_u\big)\;\;\text{for}\;\;t \in [b_u, b_u+\zeta_u)$$
and $(0,0,0)$ otherwise.
Define
$$\big(X(t),V(t), P(t)\big) = \Big(\big(X_1(t),V_1(t), P_1 (t)\big),\big(X_2(t),V_2(t), P_2 (t)\big),\ldots\Big)$$
as the process of sizes and growth rates of the living particles in the system at time $t$. 
As for the fragmentation process we have an identity between point measures
\begin{equation} \label{identity point measures}
\sum_{i = 1}^\infty {\bf 1}_{\{X_i(t)>0\}}\delta_{(X_i(t),V_i(t), P_i (t))} = \sum_{u \in {\mathcal U}}{\bf 1}_{\{b_u \leq t < b_u +\zeta_u\}}\delta_{(\xi_t^u,\tau_t^u, p_t^u)}
\end{equation}
where $\delta$ denotes the Dirac mass.

In the sequel, the following basic assumption is in force
\begin{assumption}[Basic assumption on $B$ and $\rho$] \label{basic assumption}
The division rate $x \leadsto B(x)$ is continuous. We have $B(0)=0$ and 
$\int^\infty x^{-1}B(x)dx=\infty$. The Markov kernel $\rho_0 (v,dv')$ and $\rho_1 (v,dv')$ are
 defined
on a compact set ${\mathcal E} \subset (0,\infty)$.
\end{assumption}
\begin{prop} \label{existence processus}
Work under Assumption \ref{basic assumption}. The law of 
$$\big((X(t),V(t),P(t)), t \geq 0\big)\;\;\text{or}\;\;\big((\xi_u, \tau_u, p_u), u \in {\mathcal U}\big)\;\;\text{or}\;\;\big((\xi_t^u, \zeta_t^u, p_t^u), t \geq 0, u \in {\mathcal U}\big)$$ is well-defined on an appropriate probability space.
\end{prop}
If $\mu$ is a probability measure on the state space ${\mathcal S} = [0,\infty)\times {\mathcal E}$, we shall denote indifferently by $\PP_\mu$ the law of any of the three processes above 
where the root $(\xi_{\emptyset}, \tau_{\emptyset},p_{\emptyset})$ has distribution $\mu$. The demonstration to prove the Proposition \ref{existence processus} is classic (see for instance \cite{Bertoin} and the references therein). It is a generalization of Proposition 1 of \cite{hof}.

\subsection{A many-to-one formula via a tagged branch} \label{a many-to-one formula via a tagged branch}

We will adopt the notations introduced in \cite{hof}, while adapting them to the broader context under consideration. For $u \in {\mathcal U}$, let $m^i u$ denote the $i$-th ancestor in the genealogy of $u$. We then define $$\overline{\tau_t^u} = \sum_{i = 1}^{|u|}\tau_{m^i u}\zeta_{m^i u}+\tau_t^u(t-b_u)\;\;\text{for}\;\;t\in [b_u, b_u + \zeta_u)$$ and set it to $0$ otherwise, representing the accumulated growth rate along its ancestors up to time $t$.

Let $n(u)$ denote the number of 1's in $u$, corresponding to the number of new pole bacteria in $u$'s genealogy. Similarly, $o(u)$ denotes the number of 0's in $u$, representing the number of old pole bacteria in the genealogy of $u$. In the same spirit as tagged fragments in fragmentation processes (see Bertoin's book \cite{Bertoin} for instance), we randomly select a branch in the genealogical tree: for any $k \geq 1$, if $\vartheta_k$ denotes the node of the tagged branch at the $k$-th generation, we have $$\PP(\vartheta_k = u)=\theta_0^{o(u)}\theta_1^{n(u)}\;\;\text{for every}\;\;u\in {\mathcal U},$$and $0$ otherwise.

For $t\geq 0$, the relation
$$b_{\vartheta_{C_t}}\leq t<b_{\vartheta_{C_t}}+\zeta_{\vartheta_{C_t}}$$
uniquely defines a counting process $(C_t, t \geq 0)$ with $C_0=0$. This process $C_t$ tracks the number of divisions undergone by the tagged bacterium. We can similarly define the number of divisions resulting in a bacterium with a new pole as $C_t^n := n(\vartheta_{C_t})$ and with an old pole as $C_t^o := o(\vartheta_{C_t})$. Naturally, for all $t \in \mathbb{R}$, we have $C_t = C_t^n + C_t^o$.

The process $C_t$ then allows us to define a tagged process for size, growth rate, accumulated growth rate and type as follows:
$$
\big(\chi(t), {\mathcal V}(t), \overline{{\mathcal V}}(t), Q(t)\big) = \Big(\xi_t^{\vartheta_{C_t}}, \tau_{t}^{\vartheta_{C_t}}, \overline{\tau_t^{\vartheta_{C_t}}}, p_{t}^{\vartheta_{C_t}}\Big)\;\;\text{for}\;\;t \in [b_{\vartheta_{C_t}},b_{\vartheta_{C_t}}+\zeta_{\vartheta_{C_t}})
$$
and $0$ otherwise. We then have the representation \begin{equation} \label{rep chi} \chi(t) = xe^{\overline{{\mathcal V}}(t)}\theta_0^{C_t^o}\theta_1^{C_t^n}.
\end{equation} Since the $\tau$ has a value in ${\mathcal E}$ which is a compact space there exist two positive real $e_{\min}<e_{\max}$ such that ${\mathcal V}(t) \in [e_{\min}, e_{\max}]$, it follows that \begin{equation} \label{controle croissance cumulee} e_{\min}t \leq \overline{{\mathcal V}}(t) \leq e_{\max}t. \end{equation}

The behavior of $\big(\chi(t), {\mathcal V}(t), \overline{{\mathcal V}}(t), Q(t)\big)$ can be linked to certain functionals of the entire particle system through a so-called many-to-one formula. This serves as the key tool for proving Theorem \ref{sol transport general}.

\begin{prop}[A many-to-one formula] \label{many-to-one} Under Assumption \ref{basic assumption}, for $x \in (0,\infty)$, let $\PP_x$ be defined as in Lemma \ref{counting property}. For every $t\geq 0$, the following holds: $$\E_x\big[\phi\big(\chi(t), {\mathcal V}(t), \overline{{\mathcal V}}(t)\big)\big] = \E_x\Big[\sum_{u \in {\mathcal U}}\xi_t^u\frac{e^{-\overline{\tau_t^u}}}{x}\phi\big(\xi_t^u, \tau_t^u, \overline{\tau_t^u}\big)\Big],$$
for every function $\phi: {\mathcal S}\times [0,\infty) \rightarrow [0,\infty)$. \end{prop}

\begin{proof}[Proof of Proposition \ref{many-to-one}]
For $v \in {\mathcal U}$, set $I_v = [b_{v}, b_v+\zeta_v)$. By representation \eqref{rep chi}, we have
\begin{align*}
\E_{x}\big[\phi\big(\chi(t), {\mathcal V}(t), \overline{{\mathcal V}}(t)\big)\big] & = \E_{x}\big[\phi\big(xe^{\overline{{\mathcal V}}(t)}\theta_0^{C_t^o}\theta_1^{C_t^n}, {\mathcal V}(t), \overline{{\mathcal V}}(t)\big)\big] \\
& =  \E_{x}\Big[\sum_{v \in {\mathcal U}}\phi\big(xe^{\overline{\tau_t^v}}\theta_0^{o(v)}\theta_1^{n(v)}, \tau_t^v, \overline{\tau_t^v}\big){\bf 1}_{\{t\in I_v, v=\vartheta_{C_t}\}}\Big].
\end{align*}
Introduce the discrete filtration ${\mathcal H}_n$ generated by $(\xi_u, \zeta_u, \tau_u)$ for every $u$ such that $|u|\leq n$. Conditioning with respect to ${\mathcal H}_{|v|}$ and noting that on $\{t \in I_v\}$, we have
$$\PP\big(\vartheta_{C_t} = v\,|\,{\mathcal H}_{|v|}\big)=\theta_0^{o(v)}\theta_1^{n(v)} = \frac{\xi_v e^{-\overline{\tau_{b_v}^v}}}{x},$$
we derive
\begin{align*}
&\E_{x}\Big[\sum_{v \in {\mathcal U}}\phi\big(xe^{\overline{\tau_t^v}} \theta_0^{o(v)}\theta_1^{n(v)}, \tau_t^v, \overline{\tau_t^v}\big){\bf 1}_{\{t\in I_v, v=\vartheta_{C_t}\}}\Big]\\  &= \,\E_{x}\Big[\sum_{v \in {\mathcal U}}\xi_v \frac{e^{-\overline{\tau_{b_v}^v}}}{x}\phi\big(xe^{\overline{\tau_t^v}}\theta_0^{o(v)}\theta_1^{n(v)}, \tau_t^v, \overline{\tau_t^v}\big){\bf 1}_{\{t\in I_v\}}\Big] \\&
=\, \E_x\Big[\sum_{u \in {\mathcal U}}\xi_t^u\frac{e^{-\overline{\tau_t^u}}}{x}\phi\big(\xi_t^u, \tau_t^u, \overline{\tau_t^u}\big)\Big].
\end{align*}
\end{proof}

\subsection{The behaviour of the mean empirical measure}
Denote by ${\mathcal C}_0^1({\mathcal S})$ the set of real-valued test functions with compact support in the interior of ${\mathcal S}$. 
\begin{theorem}[Behaviour of the empirical mean] \label{sol transport general}
Work under Assumption \ref{basic assumption}. Let $\mu$ be a probability distribution on ${\mathcal S}$. Define the distribution $n(t,dx,dv)$ by
$$\langle n(t,\cdot),\phi \rangle = \E_{\mu}\Big[\sum_{i = 1}^\infty\phi\big(X_i(t), V_i(t)\big)\Big]\;\;\text{for every}\;\;\phi \in {\mathcal C}^1_0({\mathcal S}).$$
Then $n(t,\cdot)$ solves (in a weak sense)
\begin{equation} \label{transport variabilite}
\left\{
\begin{array}{rl}
& \partial_t n(t,x,v) + v\, \partial_x\big(x n(t,x, v)\big)+B(x)n(t,x,v) \\ \\
 & =\; \int_{{\mathcal E}}\frac{\phi(x,v',0)}{\theta_0^2}\rho_0(v,dv')B(x/\theta_0)n(t,x/\theta_0,dv')\\ \\&+\frac{\phi(x,v',1)}{\theta_1^2}\rho_1(v,dv')B(x/\theta_1)n(t,x/\theta_1,dv'), \\ \\
 & n(0,x,v)= n^{(0)}(x,v), x \geq 0.
\end{array}
\right.
\end{equation}
with initial condition $n^{(0)}(dx,dv) = \mu(dx,dv)$. 
\end{theorem}
The demonstration generalizes this fact in \cite{hof}, which was inspired by Bertoin \cite{Bertoin} and Haas \cite{Haas} and it is made in the Section \ref{proof}. Alternative approaches to the same kind of questions include the probabilistic studies of Chauvin {\it et al.} \cite{CRW}, Bansaye {\it et al.} \cite{BDMV} or  Harris and Roberts \cite{HR} and the references therein.

\section{Statistical applications} \label{statistical analysis}

To underscore the practical relevance of our findings, I will revisit the results presented in \cite{hof}, which were originally demonstrated within a more constrained framework than the one we explore in this paper. For conciseness, the results will be briefly summarized here, with full details available in the original work. I will conclude by discussing our ongoing collaboration with Benoîte de Saporta and Bertrand Cloez, where we are extending these results to the broader context addressed in this paper.

The theories and propositions previously established applied to a specific case involving a single type of bacterium that divided symmetrically into two equal parts \cite{hof}. The use of probabilistic methods enabled the demonstration of statistical results concerning the non-parametric estimation of $B$. Specifically, it
was shown that assuming, additionally, that ${\mathcal P}_B$ possesses an invariant probability measure $\nu_B(d\boldsymbol{x})$, that is, a solution to

\begin{equation} \label{def mesure invariante} \nu_B{\mathcal P}_B = \nu_B, \end{equation}
where
$$\mu{\mathcal P}_B(d\boldsymbol{y})=\int_{{\mathcal S}}\mu(d\boldsymbol{x}){\mathcal P}_B(\boldsymbol{x},d\boldsymbol{y})$$
represents the left action of positive measures $\mu(d\boldsymbol{x})$ on $\mathcal{S}$ under the transition ${\mathcal P}_B$.

\begin{prop}\cite{hof} Under Assumption 1 in \cite{hof}, the operator ${\mathcal P}_B$ indeed admits an invariant probability measure $\nu_B$ of the form $\nu_B(d\boldsymbol{x}) = \nu_B(x, dv) dx$. Moreover, we have the representation

\begin{equation} \label{representation cle} \nu_B(y) = \frac{B(2y)}{y} \E_{\nu_B}\Big[\frac{1}{\tau_{u^-}}{\bf 1}_{{\xi{u^-} \leq 2y,;\xi_u \geq y}}\Big], \end{equation}
where $\E_{\nu_B}[\cdot]$ denotes the expectation when the initial condition $(\xi_\emptyset, \tau_\emptyset)$ follows the distribution $\nu_B$, and $\nu_B(y) = \int_{{\mathcal E}}\nu_B(y, dv')$ represents the marginal density of the invariant probability measure $\nu_B$ with respect to $y$. \end{prop}

By inverting \eqref{representation cle} and applying an appropriate change of variables, we obtain

\begin{equation} \label{first rep} B(y) = \frac{y}{2}\frac{\nu_B(y/2)}{\E_{\nu_B}\Big[\frac{1}
{\tau_{u^-}}{\bf 1}_{{\xi{u^-} \leq y, ;\xi_u \geq y/2}}\Big]}, \end{equation}
provided the denominator remains positive. This representation \eqref{first rep} suggests an estimation procedure, where the marginal density $\nu_B(y/2)$ and the expectation in the denominator are replaced by their empirical counterparts. To implement this, select a kernel function

$$K: [0,\infty)\rightarrow \R,\;\;\int_{[0,\infty)}K(y)dy=1,$$ 
and define $K_h(y) = h^{-1}K\big(h^{-1}y\big)$ for $y \in [0,\infty)$ and $h > 0$. Our estimator is then given by

\begin{align} \widehat B_n(y) & = \frac{y}{2}\frac{n^{-1}\sum_{u \in {\mathcal U}_n} K_h(\xi_u - y/2)}{n^{-1}\sum{u \in {\mathcal U}_n} \frac{1}{\tau{u^-}} {\bf 1}_{\displaystyle {\xi{u^-} \leq y, \xi_u \geq y/2}} \bigvee \varpi}, \label{def estimator} \end{align}
where $\varpi > 0$ is a threshold ensuring that the estimator is well-defined in all cases, and $x \bigvee y = \max{(x, y)}$. Thus, $(\widehat B_n(y), y \in {\mathcal D})$ is determined by the choice of the kernel $K$, the bandwidth $h > 0$, and the threshold $\varpi > 0$.

\begin{assumption} \label{prop K} The function $K$ has compact support, and for some integer $n_0 \geq 1$, we have
$\int_{[0,\infty)}x^kK(x)dx={\bf 1}_{\{k=0\}}\;\;\text{for}\;\;0 \leq k\leq n_0.$
 \end{assumption}

 For $s > 0$, where $s = \lfloor s\rfloor + {s}$ with $0 < {s} \leq 1$ and $\lfloor s\rfloor$ an integer, consider the Hölder space ${\mathcal H}^s({\mathcal D})$ of functions $f: {\mathcal D} \rightarrow \R$ that possess a derivative of order $\lfloor s \rfloor$ satisfying

\begin{equation} \label{def sob} |f^{\lfloor s \rfloor}(y) - f^{\lfloor s \rfloor}(x)| \leq c(f) |x - y|^{{s}}. \end{equation}

The minimal constant $c(f)$ for which \eqref{def sob} holds defines a semi-norm $|f|_{{\mathcal H}^s(\mathcal{D})}$. We equip the space ${\mathcal H}^s(\mathcal D)$ with the norm
$$\|f\|_{{\mathcal H}^s(\mathcal D)} = \|f\|_{L^\infty({\mathcal D})} + |f|_{{\mathcal H}^s({\mathcal D})}$$ and introduce the corresponding Hölder balls
$${\mathcal H}^s({\mathcal D}, M) = \{B,\;\|B\|_{{\mathcal H}^s({\mathcal D})} \leq M\},\;M>0.$$

For $\lambda > 0$ and a vector of positive constants $\mathfrak{c} = (r, m, \ell, L)$, consider the class ${\mathcal F}^\lambda(\mathfrak{c})$ of continuous functions $B: [0, \infty) \rightarrow [0, \infty)$ such that

\begin{equation} \label{loc control} \int_{0}^{r/2} x^{-1} B(2x) dx \leq L, \; \int_{r/2}^{r} x^{-1} B(2x) dx \geq \ell, \end{equation}
and

\begin{equation} \label{poly control} B(x) \geq m  x^\lambda \;
\text{for}\; x \geq r. \end{equation}

\begin{theorem} \cite{hof}\label{upper bound}
 Under Assumption 3 in \cite{hof} for the sparse tree case and Assumption 4 in \cite{hof} for the full tree case, consider $\widehat{B}$ defined with a kernel $K$ satisfying Assumption 2 in \cite{hof} for some $n_0 > 0$, with

$$h=c_0n^{-1/(2s+1)},\;\;\varpi_n = (\log n)^{-1}.$$

For every $M > 0$, there exist constants $c_0 = c_0(\mathfrak{c}, M)$ and $d(\mathfrak{c}) \geq 0$ such that for every $0 < s < n_0$ and any compact interval ${\mathcal D} \subset (d(\mathfrak{c}), \infty)$ with $\inf {\mathcal D} \geq r/2$, we have

$$
\sup_{\rho, B}\E_{\mu}\big[\|\widehat B_n-B\|_{L^2({\mathcal D})}^2\big]^{1/2} \lesssim (\log n)n^{-s/(2s+1)},$$
where the supremum is taken over
$$\rho \in  {\mathcal M}(\rho_{\min},\rho_{\max})\;\;\text{and}\;\;B \in {\mathcal F}^\lambda(\mathfrak{c}) \cap {\mathcal H}^s({\mathcal D}, M),$$
and $\E_{\mu}[\cdot]$ denotes expectation with respect to any initial distribution $\mu(d\boldsymbol{x})$ for $(\xi_\emptyset, \tau_\emptyset)$ on ${\mathcal S}$ such that $\int_{\mathcal S}{\mathbb V}(\boldsymbol{x})^2\mu(d\boldsymbol{x}) < \infty$. \end{theorem}
 
 In a forthcoming paper with Benoîte de Saporta and Bertrand Cloez, we aim to extend these results by considering bacteria with two distinct poles: old poles and new poles. Additionally, we will provide an adaptive and min-max estimation method. Our approach builds on the findings presented in \cite{MR4303887}, which focuses on estimating the jump rate for a class of piecewise deterministic Markov processes, including marked bacteria. We will adapt these methods to account for the branching process characteristic of bacterial populations. In the writing paper and in \cite{hof} there are simulations on real E. coli bacteria data that have been carried out.

\section{Proof of Theorem \ref{sol transport general}}\label{proof}
We fix $x \in (0,\infty)$ and first prove the result for an initial measure $\mu_x$ as in Proposition \ref{many-to-one}. Let $\phi \in {\mathcal C}^1_0({\mathcal S})$ be nonnegative. By \eqref{identity point measures} we have
$$
\langle n(t,\cdot), \phi \rangle  = \E_x\big[\sum_{i=1}^\infty\phi\big(X_i(t),Z_i(t), P_i (t)\big)\big] 
 = \E_x\big[\sum_{u \in {\mathcal U}}\phi(\xi_t^u,\tau_t^u, p_t)\big] 
$$
and applying Proposition \ref{many-to-one}, we derive 
\begin{equation} \label{n via many-to-one}
\langle n(t,\cdot), \phi \rangle =x\, \E_x\Big[\phi\big(\chi(t), {\mathcal V}(t), Q(t)\big)\frac{e^{\overline{{\mathcal V}}(t)}}{\chi(t)}\Big].
\end{equation}
For $h>0$, introduce the difference operator 
$$\Delta_h f(t) = h^{-1}\big(f(t+h)-f(t)\big).$$
We plan to study the convergence of $ \Delta_h\langle n(t,\cdot),\phi \rangle$ as $h\rightarrow 0$ using representation \eqref{n via many-to-one} in restriction to the events $\{C_{t+h}-C_t=i\}$, for $i=0,1$ and $\{C_{t+h}-C_t \geq 2\}$.
Denote by ${\mathcal F}_t$ the filtration generated by the tagged branch $\big(\chi(s), {\mathcal V}(s), Q(s),s \leq t\big)$. 
\begin{lemma} \label{counting property}
Assume that $B$ is continuous.  Let $x\in (0,\infty)$ and let $\mu_x$ be a probability measure on ${\mathcal S}$ such that $\mu_x(\{x\}\times {\mathcal E}\times \{0,1\})=1$. Abbreviate $\PP_{\mu_x}$ by $\PP_x$. For small $h>0$, we have
$$\PP_x(C_{t+h}-C_t = 1\,|\,{\mathcal F}_t)=B\big(\chi(t)\big)h+h\,\varepsilon(h),$$
with the property $|\varepsilon(h)| \leq \epsilon(h) \rightarrow 0$ as $h \rightarrow 0$, for some deterministic $\epsilon(h)$, and
$$\PP_x(C_{t+h}-C_t \geq 2) \lesssim h^2.$$
\end{lemma}

\begin{proof}

Using the same methodology as in Lemma 1 of \cite{hof}, first, we observe that $$\{C_{t+h}-C_t \geq 1\} = \{t < b_{\vartheta_{C_t}}+\zeta_{\vartheta_{C_t}}\leq t+h\}.$$ Furthermore, since $\xi_{\vartheta_{C_t}} = x\exp\left(\overline{{\mathcal V}}(b_{\vartheta_{C_t}})\right)\theta_0^{\vartheta_{C_t^0}}\theta_1^{\vartheta_{C_t^1}}$, it follows from \eqref{def div rate} that \begin{align*}
&\PP(C_{t+h}-C_t \geq 1\,|\,{\mathcal F}_t) \\
=& \int_{t-b_{\vartheta_{C_t}}}^{t+h-b_{\vartheta_{C_t}}}B\Big(xe^{\overline{{\mathcal V}}(b_{\vartheta_{C_t}})+s{\mathcal V}(s)}\theta_0^{C_t^0}\theta_1^{C_t^1}\Big)\exp\Big(-\int_0^s B\Big(xe^{\overline{{\mathcal V}}(b_{\vartheta_{C_t}})+s'{\mathcal V}(s')}\theta_0^{C_t^0}\theta_1^{C_t^1}
\Big)ds'\Big)ds.
\end{align*}

By introducing the term $B\left(xe^{\overline{{\mathcal V}}(b_{\vartheta_{C_t}}) + {\mathcal V}(t)(t - b_{\vartheta_{C_t}})}\theta_0^{C_t^0}\theta_1^{C_t^1}\right)$ into the integral, and recognizing that $\overline{{\mathcal V}}(b_{\vartheta_{C_t}}) + {\mathcal V}(t)(t - b_{\vartheta_{C_t}}) = \overline{{\mathcal V}}(t)$, we arrive at the first part of the lemma, utilizing the representation \eqref{rep chi} and the uniform continuity of $B$ over compact domains.

For the second part, consider the $({\mathcal F}t)$-stopping time $$\Upsilon_t = \inf\{s>t, C_s - C_t \geq 1\}$$ and note that ${C_{t+h} - C_t \geq 1} = {\Upsilon_t \leq t + h} \in {\mathcal F}_{\Upsilon_t}$. Then, by expressing $$\{C_{t+h}-C_t \geq 2\} = \{\Upsilon_t < t+h,\;\Upsilon_{\Upsilon_{t}} \leq t+h\},$$ and conditioning on ${\mathcal F}_{\Upsilon_t}$, we obtain\begin{align*}
& \PP(C_{t+h}-C_t \geq 2) \\
= &\; \E\Big[\int_t^{t+h-\Upsilon_t}B\Big(xe^{\overline{{\mathcal V}}(b_{\vartheta_{C_t}})+s{\mathcal V}(s)}\theta_0^{C_t^0}\theta_1^{C_t^1}\Big)e^{-\int_0^s B\big(xe^{\overline{{\mathcal V}}(b_{\vartheta_{C_t}})+s'{\mathcal V}(s')}\theta_0^{C_t^0}\theta_1^{C_t^1}
\big)ds'}ds{\bf 1}_{\{\Upsilon_t < t+h\}}
\Big] \\
\leq &\;
h \sup_{y \leq x \exp (\max(1/\theta_0, 1/\theta_1)e_{\max} t)} 
B(y)\,\PP(\Upsilon_t < t+h).
\end{align*}
Similarly, $\PP(\Upsilon_t < t + h) \lesssim h$, leading to the desired conclusion.

\end{proof}

Since $\phi \in {\mathcal C}^1_0({\mathcal S})$, there exists $d(\phi)>0$ such that $\phi(y,v)=0$ if $y \geq d(\phi)$. By \eqref{controle croissance cumulee}, we infer 
\begin{equation} \label{integrabilite reste}
\Big|\phi\big(\chi(t), {\mathcal V}(t)\big)\frac{e^{\overline{{\mathcal V}}(t)}}{\chi(t)}\Big| \leq \sup_{y,v}\phi(y,v)\frac{\exp(e_{\max}t)}{d(\phi)}
\end{equation}
By Lemma \ref{counting property} and \eqref{integrabilite reste}, we derive
\begin{equation} \label{first estimate poisson}
\E_x\Big[\Delta_h\Big(\phi\big(\chi(t), {\mathcal V}(t)\big)\frac{e^{\overline{{\mathcal V}}(t)}}{\chi(t)}\Big){\bf 1}_{\{C_{t+h}-C_t \geq 2\}}\Big] \lesssim h.
\end{equation}
On the event $\{C_{t+h}-C_t=0\}$, the process ${\mathcal V}(s)$ is constant for $s \in [t,t+h)$ and so is $\frac{e^{\overline{{\mathcal V}}(s)}}{\chi(s)}$ thanks to \eqref{rep chi}. It follows that
$$\Delta_h\Big(\phi\big(\chi(t), {\mathcal V}(t)\big)\frac{e^{\overline{{\mathcal V}}(t)}}{\chi(t)}\Big)= \Delta_h\phi\big(\chi(t), {\mathcal V}(s)\big)_{\big|_{s=t}}\frac{e^{\overline{{\mathcal V}}(t)}}{\chi(t)}$$
on $\{C_{t+h}-C_t=0\} $ and also
$$
\Big|\Delta_h\phi\big(\chi(t), {\mathcal V}(s)\big)_{\big|_{s=t}}\frac{e^{\overline{{\mathcal V}}(t)}}{\chi(t)}\Big| \leq \sup_{y,v}|\partial_y\phi(y,v)|xe_{\max}  \frac{\exp(2e_{\max}t)}{d(\phi)}
$$
on $\{C_{t+h}-C_t=0\}$ likewise. Since $\PP_x(C_{t+h}-C_t=0)\rightarrow 1$ as $h\rightarrow 0$, by dominated convergence
\begin{align}
& \;x\,\E_x\Big[\Delta_h\Big(\phi\big(\chi(t), {\mathcal V}(t)\big)\frac{e^{\overline{{\mathcal V}}(t)}}{\chi(t)}\Big){\bf 1}_{\{C_{t+h}-C_t =0\}}\Big] \nonumber \\
\rightarrow &\; x\,\E_x\big[\partial_1\phi\big(\chi(t), {\mathcal V}(t)\big){\mathcal V}(t)e^{\overline{{\mathcal V}}(t)}\big]\;\;\text{as}\;\;h\rightarrow 0.
\label{convergence poisson 2}
\end{align}
By Proposition \ref{many-to-one} again, this last quantity is equal to $\langle n(t,dx,dv), xv\,\partial_x \phi\rangle$. On $\{C_{t+h}-C_t=1\}$, we successively have
$$(\chi(t+h), Q(t+h)) = \theta_{Q(t+h)}\chi(t)+\varepsilon_1(h),$$
$$\phi\big(\chi(t+h), {\mathcal V}(t+h),Q(t+h)\big) = \phi\big(\chi(t)\theta_{Q(t+h)}, {\mathcal V}(t+h)\big)+\varepsilon_2(h)$$
and
$$\exp\big(\overline{{\mathcal V}}(t+h)\big) = \exp\big(\overline{{\mathcal V}}(t)\big)+\varepsilon_3(h)$$
with the property $|\varepsilon_i(h)|\leq \epsilon_1(h) \rightarrow 0$ as $h\rightarrow 0$, where $\epsilon_1(h)$ is deterministic, thanks to \eqref{rep chi} and \eqref{controle croissance cumulee}. Moreover, 
$${\mathcal V}(t+h) = \tau_{\vartheta_{C_t+1}}\;\;\text{on}\;\;\{C_{t+h}-C_t=1\}.$$
It follows that
\begin{align*}
& \E_x\Big[\phi\big(\chi(t+h), {\mathcal V}(t+h),Q(t+h)\big)\frac{e^{\overline{{\mathcal V}}(t+h)}}{\chi(t+h)}{\bf 1}_{\{C_{t+h}-C_t =1\}}\Big] \\
= &\; \E_x\Big[\phi\big(\chi(t)\theta_{\vartheta_{C_t +1}}, \tau_{\vartheta_{C_t+1}}, p_{\vartheta_{C_t +1}}\big)\frac{e^{\overline{{\mathcal V}}(t)}}{\chi(t)\theta_{\vartheta_{C_t +1}}}{\bf 1}_{\{C_{t+h}-C_t =1\}}\Big]+\epsilon_2(h) \\
= &\; \E_x\Big[\phi\big(\chi(t)\theta_{\vartheta_{C_t +1}}, \tau_{\vartheta_{C_t+1}},p_{\vartheta_{C_t +1}}\big)\frac{e^{\overline{{\mathcal V}}(t)}}{\chi(t)\theta_{\vartheta_{C_t +1}}}{\bf 1}_{\{C_{t+h}-C_t \geq 1\}}\Big]+\epsilon_3(h)
\end{align*}
where $\epsilon_2(h), \epsilon_3(h)\rightarrow 0$ as $h\rightarrow 0$, and where we used the second part of Lemma \ref{counting property} in order to obtain the last equality. Conditioning with respect to ${\mathcal F}_t \bigvee \tau_{\vartheta_{C_t+1}}$ and using that
$\{C_{t+h}-C_t \geq 1\}$ and $\tau_{\vartheta_{C_t+1}}$ are independent, applying the first part of Lemma \ref{counting property}, this last term is equal to
\begin{align*}
&  \E_x\Big[\phi\big(\chi(t)\theta_{\vartheta_{C_t +1}}, \tau_{\vartheta_{C_t+1}}, p_{\vartheta_{C_t +1}}\big)\frac{e^{\overline{{\mathcal V}}(t)}}{\theta_{\vartheta_{C_t +1}}\chi(t)}B\big(\chi(t)\big)h\Big] + \epsilon_3(h)\\
=\;&  \E_x\Big[\int_{{\mathcal E}}\phi\big(\chi(t)\theta_0 , v',0\big)\rho_0\big({\mathcal V}(t), dv'\big)\frac{e^{\overline{{\mathcal V}}(t)}}{\theta_0\chi(t)}B\big(\chi(t)\big)h\Big] \\&+ \E_x\Big[\int_{{\mathcal E}}\phi\big(\chi(t)\theta_1 , v',1\big)\rho_1\big({\mathcal V}(t), dv'\big)\frac{e^{\overline{{\mathcal V}}(t)}}{\theta_1\chi(t)}B\big(\chi(t)\big)h\Big] + \epsilon_4(h)
\end{align*}
where $\epsilon_4(h)\rightarrow 0$ as $h \rightarrow 0$. Finally, using Lemma \ref{counting property} again, we derive
\begin{align}
& \E_x\Big[\Delta_h\Big(\phi\big(\chi(t), {\mathcal V}(t), Q(t)\big)\frac{e^{\overline{{\mathcal V}}(t)}}{\chi(t)}\Big){\bf 1}_{\{C_{t+h}-C_t =1\}}\Big] \nonumber \\
\rightarrow &\;\E_x\Big[\Big(\int_{{\mathcal E}}\frac{\phi\big(\chi(t)\theta_0 , v',0\big)\rho_0\big({\mathcal V}(t), dv'\big)}{\theta_0}+\frac{\phi\big(\chi(t)\theta_1 , v',1\big)\rho_1\big({\mathcal V}(t), dv'\big)}{\theta_1}-\phi\big(\chi(t),{\mathcal V}(t)\big)\Big)\\ &\times\frac{e^{\overline{{\mathcal V}}(t)}}{\chi(t)}B\big(\chi(t)\big)\Big]
\label{convergence poisson 1}
\end{align}
as $h \rightarrow 0$. By Proposition \ref{many-to-one}, this last quantity is equal to
$$\big\langle n(t,dx,dv), \big(\int_{{\mathcal E}}\frac{\phi(x\theta_0,v',0)}{\theta_0}\rho_0(v,dv')+\frac{\phi(x\theta_1,v',1)}{\theta_1}\rho_1(v,dv')-\phi(x,v)\big)B(x)\big\rangle$$
which, in turn,  is equal to 
\begin{align*}
&\big\langle n(t,2dx,dv),\int_{{\mathcal E}}\frac{\phi(x,v',0)}{\theta_0^2}\rho_0(v,dv')B(x/\theta_0)\\&
+\frac{\phi(x,v',1)}{\theta_1^2}\rho_1(v,dv')B(x/\theta_1)\big\rangle-\big\langle n(t,dx,dv),\phi(x,v)B(x)\big\rangle\end{align*}
by a simple change of variables. Putting together the estimates \eqref{first estimate poisson}, \eqref{convergence poisson 2} and \eqref{convergence poisson 1}, we conclude
\begin{align*}
&\partial_t \langle n(t,dx,dv), \phi\rangle - \langle n(t,dx,dv), xv\partial_x \phi\rangle + \langle n(t,dx,dv)B(x),\phi\rangle \\
 = & \;  \big\langle n(t,2dx,dv),\int_{{\mathcal E}}\frac{\phi(x,v',0)}{\theta_0^2}\rho_0(v,dv')B(x/\theta_0)+\frac{\phi(x,v',1)}{\theta_1^2}\rho_1(v,dv')B(x/\theta_1)\big\rangle\big\rangle,
\end{align*}
which is the dual formulation of \eqref{transport variabilite}. The proof is complete.


\begin{thebibliography}{99}

\bibitem{Sturm}
M.~Arnoldini A. Benecke M. Ackermann M. Benz J.~Dormann A.~Sturm, M.~Heinemann
  and W.-D. Hardt.
\newblock The cost of virulence: Retarded growth of salmonella typhimurium
  cells expressing type iii secretion system 1.
\newblock {\em PLoS Pathogens}, 7(7):10, 2011.

\bibitem{BB}
 B.~Delyon, B. de~Saporta, N.~Krell and L.~Robert.
\newblock Investigation of asymmetry in e. coli growth rate.

\bibitem{BDMV}
V.~Bansaye, J-F. Delmas, L.~Marsalle, and V.C. Tran.
\newblock Limit theorems for {M}arkov processes indexed by supercritical
  {G}alton {W}atson tree.
\newblock {\em The Annals of Applied Probability}, 21:2263--2314, 2011.

\bibitem{Bertoin}
J.~Bertoin.
\newblock {\em Random fragmentation and Coagulation Processes}.
\newblock Cambridge University Press, 2006.

\bibitem{CRW}
B.~Chauvin, A.~Rouault, and A.~Wakolbinger.
\newblock Growing conditioned trees.
\newblock {\em Stochastic Process. Appl.}, 39:117--130, 1991.



\bibitem{hof}
 M.~Doumic, M.~Hoffmann, N.~Krell and L.~Robert.
\newblock Statistical estimation of a growth-fragmentation model observed on a
  genealogical tree.
\newblock {\em Bernoulli}, 21(3):1760--1799, 2015.

\bibitem{hof2}
M.~Doumic, M.~Hoffmann, N.~Krell, L. Robert, S.~Aymerich,  and J.~Robert.
\newblock Division control in escherichia coli is based on a size-sensing
  rather than timing mechanism.
\newblock {\em DBMC Biology}, 12:17, 2014.


\bibitem{Haas}
B.~Haas.
\newblock Loss of mass in deterministic and random fragmentations.
\newblock {\em Stoch. Proc. App.}, 106:245--277, 2003.

\bibitem{HR}
S.~C. Harris and M.~I. Roberts.
\newblock The many-to-few lemma and multiple spines.
\newblock {\em arXiv:1106.4761v3}, 2012.

\bibitem{MR4303887}
N.~Krell and E.~Schmisser.
\newblock Nonparametric estimation of jump rates for a specific class of
  piecewise deterministic {M}arkov processes.
\newblock {\em Bernoulli}, 27(4):2362--2388, 2021.

\end{thebibliography}
\end{document}